\documentclass[twoside]{article}  
\usepackage{lipsum}
\usepackage{graphicx}
\usepackage{epstopdf}
\usepackage{graphicx}
\usepackage{amssymb}
\usepackage{amsthm}
\usepackage{titlesec}
\titleformat*{\section}{\large\bfseries\centering}
\newtheorem{thm}{Theorem}

\newtheorem{lem}{Lemma}
\newtheorem{pro}{Proposition}

\newcommand{\bd}{{\rm bd}}

\newcommand{\diam}{{\rm diam}}
\theoremstyle{remark}

\date{}
\newcommand\blfootnote[1]{%
  \begingroup
  \renewcommand\thefootnote{}\footnote{#1}%
  \addtocounter{footnote}{-1}%
  \endgroup}

\begin{document}

\title{\textbf{REDUCED SPHERICAL CONVEX BODIES}}

\author{MAREK LASSAK   and MICHA\L \  MUSIELAK \\
\\
Institute of Mathematics and Physics,  University of Science and \\ Technology, Kaliskiego 7, Bydgoszcz 85-789, Poland \\
e-mails: marek.lassak@utp.edu.pl, michal.musielak@gmail.com }


\maketitle

\blfootnote{\textit{Key words and phrases:} spherical convex body, spherical geometry, hemisphere, lune, width, constant width, thickness, diameter, reduced body, extreme point}\blfootnote{\textit{Mathematics Subject Classification:} Primary 52A55; Secondary 97G60} 

\begin{abstract} The aim of this paper is to present some properties of reduced spherical convex bodies on the two-dimensional sphere $S^2$. 
The intersection of two different non-opposite hemispheres 
is called a lune. 
By its thickness we mean the distance of the centers of the two semicircles bounding it. The thickness $\Delta (C)$ of $C$ is the minimum thickness of a lune containing $C$. We say that a spherical convex body $R$ is reduced if $\Delta (Z) < \Delta (R)$ for every spherical convex body $Z \subset R$ different from $R$. 
Our main theorem permits to describe the shape of reduced bodies of thickness below $\frac{\pi}{2}$. 
It implies a number of corollaries. 
In particular, we estimate the diameter of reduced spherical bodies in terms of their thickness. Reduced bodies of thickness at least $\frac{\pi}{2}$ have constant width. 
Spherical convex bodies of constant width below $\frac{\pi}{2}$ are strictly convex.
\end{abstract}

\section{Introduction}
\label{intro}

\noindent
Let $S^d$ be the unit sphere of 
the $(d+1)$-dimensional Euclidean space $E^{d+1}$. 
By a {\it great circle} of $S^d$ we mean the intersection of $S^d$ with any two-dimensional subspace of $E^{d+1}$. 
The common part of $S^d$ with any hyper-subspace of $E^{d+1}$ is called a {\it $(d-1)$-dimensional great sphere} of $S^d$.
In particular, for $S^2$ it is nothing else but a great circle.
By a pair of {\it antipodes} of $S^d$ we mean any pair of points which are  obtained as the intersection of $S^d$ with a one-dimensional subspace of $E^{d+1}$.
Observe that if two different points are not antipodes, there is exactly one great circle containing them.
If different points $a, b \in S^d$ are not antipodes, by the {\it spherical arc}, or shortly {\it arc}, $ab$ connecting them we understand the shorter part of the great circle containing $a$ and $b$. 
By the {\it spherical distance} $|ab|$, or shortly {\it distance}, of 
these points we mean the length of the arc connecting them. 
A subset of $S^d$ is called {\it convex} if it does not contain any pair of antipodes of $S^d$ and if together with every two points it contains the arc connecting them.
By {\it a spherical convex body} we mean a closed convex set with non-empty interior.
If in the boundary of a spherical convex body there is no arc, we say that the body is {\it strictly convex}.  
Convexity on $S^d$ is considered in very many papers and monographs.
For instance in \cite{DGK}, \cite{Fe1}, \cite{Fe2} \cite{FIN}, \cite{GHS}, \cite{Ha}, \cite{Le}, \cite{NS}, \cite{VB}.

The set of points of $S^d$ in the distance at most $\rho$, where $\rho \in (0, \frac{\pi}{2})$, from a point $c \in S^d$ is called a {\it spherical ball}, or shorter {\it a ball}, of {\it radius} $\rho$ and {\it center} $c$.  
Balls on $S^2$ are called {\it disks} and the boundary of a disk is called a {\it spherical circle}. 
Balls of radius $\frac{\pi}{2}$ are called {\it hemispheres}.
The hemisphere with center $m$ is denoted by $H(m)$.
In other words, a {\it hemisphere} is the intersection of $S^d$ with any closed half-space of $E^{d+1}$.
Two hemispheres whose centers are antipodes are called {\it opposite hemispheres}.

Let $p$ be a boundary point of a convex body $C \subset S^d$.
We say that a hemisphere $H$ {\it supports $C$ at} $p$ provided $C \subset H$ and $p$ belongs to the $(d-1)$-dimensional great sphere bounding $H$.
If at $p$ the body $C$ is supported by exactly one hemisphere, we say that $p$ {\it is a smooth point of} $C$. 
If all boundary points of $C$ are smooth, then $C$ is called to be {\it smooth}.

A {\it spherical $(d-1)$-dimensional ball of radius $\rho\in \left( 0, \frac{\pi}{2}\right]$}is the set of points of a $(d-1)$-dimensional great sphere of $S^d$ in the distance at most $\rho$ from a fixed point, called {\it the center} of this ball. 
The $(d-1)$-dimensional balls of radius $\frac{\pi}{2}$ are called {\it $(d-1)$-dimensional hemispheres}, and if $d=2$ we call them {\it semicircles}.

If hemispheres $G$ and $H$ are different and not opposite, then $L = G \cap H$ is called a {\it lune}. 
The two $(d-1)$-dimensional hemispheres bounding $L$ and contained in $G$ and $H$, respectively, are denoted by $G/H$ and $H/G$. 
The {\it thickness} $\Delta (L)$ of $L \subset S^d$ is defined as the distance of the centers of $G/H$ and $H/G$.  
For $d=2$ the $(d-1)$-dimensional hemispheres bounding $L$ are called the {\it semicircles bounding} $L$. 
For $d = 2$, by the {\it corners} of $L$ we understand the two points of the set $(G/H )\cap (H/G)$.

For every hemisphere $K$ supporting a convex body $C\subset S^d$ we are finding hemispheres $K^*$ supporting $C$ such that the lunes $K\cap K^*$ are of the minimum thickness (by compactness arguments at least one such a hemisphere $K^*$ exists).
The thickness of the lune $K \cap K^*$ is called {\it the width of $C$ determined by} $K$ and it is denoted by ${\rm width}_K (C)$ (see \cite{L2}).
The width of $C$ determined by $K$ changes continuously, as the position of $K$ changes (see Theorem 2 of \cite{L2}).
If for all hemispheres $K$ supporting $C$ the numbers ${\rm width}_K (C)$ are equal, we say that $C$ is {\it of constant width} (see \cite{L2}).
An application of spherical convex bodies of constant width is given in \cite{HN}.

By the {\it thickness} $\Delta (C)$ of a convex body $C \subset S^d$ we understand the minimum width of $C$ determined by $K$ over all supporting hemispheres $K$ of $C$ (see \cite{L2}).
The thickness of $C$ is nothing else but the minimum thickness of a lune containing $C$ (see  \cite{L2}); this approach is sometimes more convenient.

After \cite{L2} we say that a spherical convex body $R \subset S^d$ is {\it reduced} if $\Delta (Z) < \Delta (R)$ for every convex body $Z \subset R$ different from $R$. 
This definition is analogical to the definition of a reduced body in Euclidean space $E^d$ given in \cite{He}. 
See also \cite{Gr}, \cite{L1} and the survey article \cite{LM}.
For a larger context see Part 5.4 of \cite{HM}.
Simple examples of reduced spherical convex bodies on $S^d$ are spherical bodies of constant width and, in particular, the balls on $S^d$.
Also each of the $2^d$ parts of a spherical ball on $S^d$ dissected by $d$ great pairwise orthogonal $(d-1)$-dimensional spheres through the center of a spherical ball $B$ is a reduced spherical body.
It is called {\it $1\over 2^d$-part of a ball} (see \cite{L2}). 
In particular, for $d=2$, it is called a {\it quarter of a spherical disk}.
There is a wide class of reduced odd-gons on $S^2$ (see \cite{L3}).
In particular, spherical regular odd-gons of thickness at most $\frac{\pi}{2}$ are reduced.
Figure 2 of \cite{L3} shows a non-regular reduced pentagon.


\section{Supporting hemispheres of a spherical convex body}
\label{sec:2}

Let $C \subset S^2$ be a spherical convex body and let $X, Y, Z$ be different supporting hemispheres of $C$.
Denote by $x,y,z$ the centers of $X,Y,Z$, respectively. 
We introduce the following three-argument relationship $\prec \! XYZ$.  
In the case when there is no common point of the three great circles bounding $X, Y, Z$ and if $x, y, z$ are in this order on the boundary of the spherical triangle $xyz$ according to the positive orientation, then we write $\prec \! XYZ$.
In the case when the three great circles bounding $X, Y, Z$ have a common point $c$ and if $x, y, z$ are in this order on the boundary of the hemisphere centered at $c$, then we write $\prec \! XYZ$. 
If $\prec \!XYZ$, then we say that $X, Y, Z$ {\it support $C$ in this order}.
We need the notion $\prec \! XYZ$ as a tool in the proof of our main Theorem \ref{main}, and also for Lemma \ref{between} applied in its proof. 

By the way, $\prec \! XYZ$ if and only if $\prec \! YZX$ and if and only if $\prec \! ZXY$.

The symbol $\preceq \! XYZ$ means that $\prec \!XYZ$ or $X=Y$ or $Y=Z$ or $Z=X$.
If $\prec XYZ$ (respectively: $\preceq XYZ$), then we say that {\it $Y$ supports $C$ strictly between} (respectively: {\it between}) {\it $X$  and $Z$}.

Assume that $\prec \!XYZ$, where $X,Y,Z$ are supporting hemispheres of $C$.
Then if we go on the boundary of $C$ according to the positive orientation starting at any point of $C \cap \bd(X)$, we do not reach any point of $C \cap \bd (Z)$ before reaching all points of $C \cap \bd(Y)$.

Observe that for any boundary point $p$ of a spherical convex body $C \subset S^2$ there are hemispheres $H_r$ and $H_l$ supporting $C$ at $p$ such that every hemisphere $H$ supporting $C$ at $p$ satisfies $H_r \preceq H \preceq H_l$. 
We call $H_r$ and $H_l$ the {\it right} and, respectively, {\it the left supporting hemispheres of $C$ at} $p$. 
Clearly, if $p$ is  a smooth point of $C$, then the hemispheres $H_r$ and $H_l$ coincide.
The left and right supporting hemispheres of $C$ are called {\it extreme supporting hemispheres of} $C$.  


\begin{lem}\label{leftright} 
Let $C\subset S^2$ be a convex body and let $L = G \cap H$ be a lune of thickness $\Delta (C)$ containing $C$, where $G$ and $H$ are hemispheres. 
Denote by $g$ the center of $G/H$ and by $h$ the center of $H/G$.
We claim that at least one of the hemispheres $G$ and $H$ is a  right supporting hemisphere of $R$ (at $g$ if this is $G$ and at $h$ if this is $H$), and at least one of these hemispheres is a left supporting hemisphere of $C$ (at $g$ if this is $G$ and at $h$ if this is $H$).
\end{lem}

\medskip
The proof is analogous to the proof of Proposition 2.1 from \cite{L3}.

Having in mind the forthcoming Lemma \ref{between} and Theorem \ref{main},  following Part I of Theorem 1 of \cite{L2} recall that for every supporting hemisphere $N$ of a spherical convex body $C \subset S^2$ of thickness below $\frac{\pi}{2}$ the hemisphere $N^*$ is unique.


\begin{lem} \label{between}
Assume that  hemispheres $N_1$, $N_2$ and $N_3$ support a spherically convex body $C$ of thickness below $\frac{\pi}{2}$ and let ${\rm width}_{N_i} (C) = \Delta (C)$ for $i = 1, 2, 3$.
We have $\prec \!N_1N_2N_3$ if and only if $\prec \!N_1^*N_2^*N_3^*$.
\end{lem}

\begin{proof}
Since ${\rm width}_{N_i} (C) = \Delta (C)$, the thickness of the lune $N_i \cap N_i^*$ is $\Delta (C)$ for $i=1,2,3$.
Denote by $a_i$ the center of the semicircle $N_i/N_i^*$ and by $b_i$ the center of the semicircle $N_i^*/N_i$ for $i=1, 2, 3$. 
By Claim 2 of \cite{L2} these six centers belong to the boundary of $C$.

Denote by $o_i$ be the center of the hemisphere $N_i$ and by $o_i^*$ the center of the hemisphere $N_i^*$ for $i=1,2,3$.
Since $\Delta (C) < \frac{\pi}{2}$, we see that $o_i$ and $o_i^*$ do not belong to $C$ for $i =1,2,3$.
Thus $a_i$ and $b_i$ belong to the arc $o_io_i^*$ for $i=1,2,3$.

Put $A = \{a_1,a_2,a_3,b_1,b_2,b_3\}$ and $O = \{o_1,o_2,o_3,o_1^*, o_2^*, o_3^*\}$.
Observe that every point of $A$ belongs to any hemisphere with the center in $O$. 
Thus the distance between any point of $O$ and any point of $A$ is at most $\frac{\pi}{2}$.
Consequently, every hemisphere with the center in $A$ contains $O$.

Denote by $V$ the intersection of the six hemispheres with the centers in $A$. 
From the inclusion $C \subset V$ we conclude that $V$ has non-empty interior.
Moreover, the points of $A$ are not contained in one great circle, which implies that $V$ does not contain any pair of antipodes. 
Consequently, $V$ is a spherical convex body.
Observe that all points of $O$ belong to the boundary of $V$.

Note that the length of $o_io_i^*$ is equal $\pi-\Delta (R)$ and is greater than $\frac{\pi}{2}$  for $i=1,2,3$.
Consider the lune $L_1 = H(a_1)\cap H(b_1)$. 
By $V\subset L_1$ we conclude that all points from $O$ belong to $L_1$.
Denote by $c$ and $c'$ corners of $L_1$ and consider two spherical triangles $T=o_1co_1^*$ and $T'=o_1c'o_1^*$.
Clearly, $T\cup T' = L_1$.
Note that $\diam (T)= \diam (T') =\pi - \Delta (R)$ and the only arc of length $\pi - \Delta (R)$ contained in $T$ or contained in $T'$ is $o_1o_1^*$.
Therefore either $o_1=o_2^*$ and $o_1^*=o_2$, or exactly one of points $o_2$ and $o_2^*$ belongs to $T$ and exactly one belongs to $T'$. 
Thus either $o_1o_1^*=o_2^*o_2$ or $o_1o_1^*$ intersects $o_2o_2^*$ and the point of the intersection belongs to the interior of $V$.
In the same way we show that either two arcs from amongst $o_1o_1^*$, $o_2o_2^*$, $o_3o_3^*$ coincide, or each pair of them intersects in the interior of $V$. 
In both cases we conclude that $o_1,o_2,o_3$ are on the boundary of $V$ in this order according to the positive orientation if and only if $o_1^*,o_2^*,o_3^*$ lay on the boundary $V$ in this order according to the positive orientation.
Thus by the convexity of $V$ we obtain that $\prec \!N_1N_2N_3$ if and only if $\prec \!N_1^*N_2^*N_3^*$.
\end{proof}

\label{sec:3}

\section{Reduced bodies of thickness below $\frac{\pi}{2}$}

\vskip0.2cm
Our main Theorem \ref{main} gives a description of the shape of reduced bodies of thickness below $\frac{\pi}{2}$ on $S^2$.
It is analogous to its version in $E^2$ presented in Theorem 3 of \cite{L1}. 
We are not able to present analogical proof as this in  \cite{L1}.
The first trouble is in the lack of the notion of parallelism on the sphere, which was used many times in the proof for $E^2$.
The second trouble is that we are not able to give a spherical analog of the relationship $l_1 \prec l_2$ for a pair of different directions $l_1, l_2$ considered in \cite{L1} for $E^2$.
It appears that a similar notion $G_1 \prec G_2$ for pairs of great circles, or a notion $H_1 \prec H_2$ for pairs of hemispheres (even supporting a spherical convex body $C$) does not make sense; the reason is that every two different great circles intersect at two points (so not at one point as two different non-parallel straight lines in the plane), and thus we are not able to state which of these great circles is ``earlier".
What helps, is the notion of the three argument relationship $\prec \! XYZ$ presented in Section 2.
It permits to establish the order of any three different hemispheres $X, Y, Z$ supporting a spherical convex body $C$, which appears to be sufficient for our needs in the proof of Theorem \ref{main} and also the applied there Lemma \ref{between}.


\begin{thm} \label{main}
Let $R \subset S^2$ be a reduced spherical body with $\Delta (R) < {\pi \over 2}$.
Let $M_1$ and $M_2$ be supporting hemispheres of $R$ such that ${\rm width}_{M_1} (R) = \Delta (R) = {\rm width}_{M_2} (R)$ and ${\rm width}_{M} (R) > \Delta (R)$ for every hemisphere $M$ satisfying $\prec \!M_1MM_2$.
Consider the lunes $L_1 = M_1 \cap M_1^*$ and $L_2 = M_2 \cap M_2^*$.  
We claim that the arcs $a_1a_2$ and $b_1b_2$ are in the boundary of $R$, where $a_i$ is the center of $M_i/M_i^*$ and $b_i$ is the center of $M_i^*/M_i$ for $i=1,2$.
Moreover, $|a_1a_2| = |b_1b_2|$.
\end{thm}

\begin{proof} 
By Claim 2 of \cite{L2} points $a_1, b_1, a_2, b_2$ belong to $R$ and thus to the boundary of $R$.
If we assume that the piece of the boundary of $R$ from $a_1$ to $a_2$ (according to the positive orientation) is not an arc, then there is an extreme point $e$ of $R$ on this piece of the boundary different from $a_1$ and $a_2$. 
Then by Theorem 4 from \cite{L2} there exists a lune $L \supset R$ of thickness $\Delta(R)$ with $e$ as the center of one of the two semicircles bounding $L$. 
Denote by $M_L$ the hemisphere for which $e \in \bd (M_L)$, where $L = M_L \cap M_L^*$.
Observe that $L$ differs from $L_1$ and $L_2$. 
The reason is that if we assume $L = L_i$ for $i \in \{1,2\}$, then both $a_i$ and $e$ are the centers of the same semicircle bounding the lune $L = L_i$, and thus $a_i = e$, which is impossible.
Since $L$ differs from $L_1$ and $L_2$, and since $e$ belongs to the piece of the boundary from $a_1$ to $a_2$ being different from these points, we have $\prec \!M_1M_LM_2$.  
This, the description of $M_L$ and $\Delta(L) = \Delta(R)$ contradict the assumption of our theorem that for every hemisphere $M$ satisfying $\prec \!M_1MM_2$ we have ${\rm width}_{M} (R) > \Delta (R)$. 
Hence the considered piece of the boundary must be the arc $a_1a_2$. 

Now we intend to show that the piece of the boundary of $R$ from $b_1$ to $b_2$ (according to the positive orientation) is the arc $b_1b_2$. 
Assume the opposite. 
Then there exists an extreme point $e'$ of $R$ on this piece different from $b_1$ and $b_2$. 
Consequently, by Theorem 4 from \cite{L2} there exists a lune $L' \supset  R$ of thickness $\Delta(R)$ with $e'$ as the center of one of the two hemispheres bounding $L'$. 
Moreover, $L'$ has the form $M_{L'} \cap (M_{L'})^*$ where $\prec \!M_1^*(M_{L'})^*M_2^*$.

By Lemma \ref{between} we have $\prec \! M_1M_{L'}M_2$. 
This and ${\rm width}_{M_{L'}} (R) = \Delta (R)$ contradict the assumption of our theorem that ${\rm width}_{M} (R) > \Delta (R)$ for every hemisphere $M$  fulfilling $\prec \!M_1MM_2$.
Consequently the considered piece of the boundary of $R$ is the arc $b_1b_2$.

Finally, show that $|a_1a_2| = |b_1b_2|$.

\begin{center}

\includegraphics[width=3.8in]{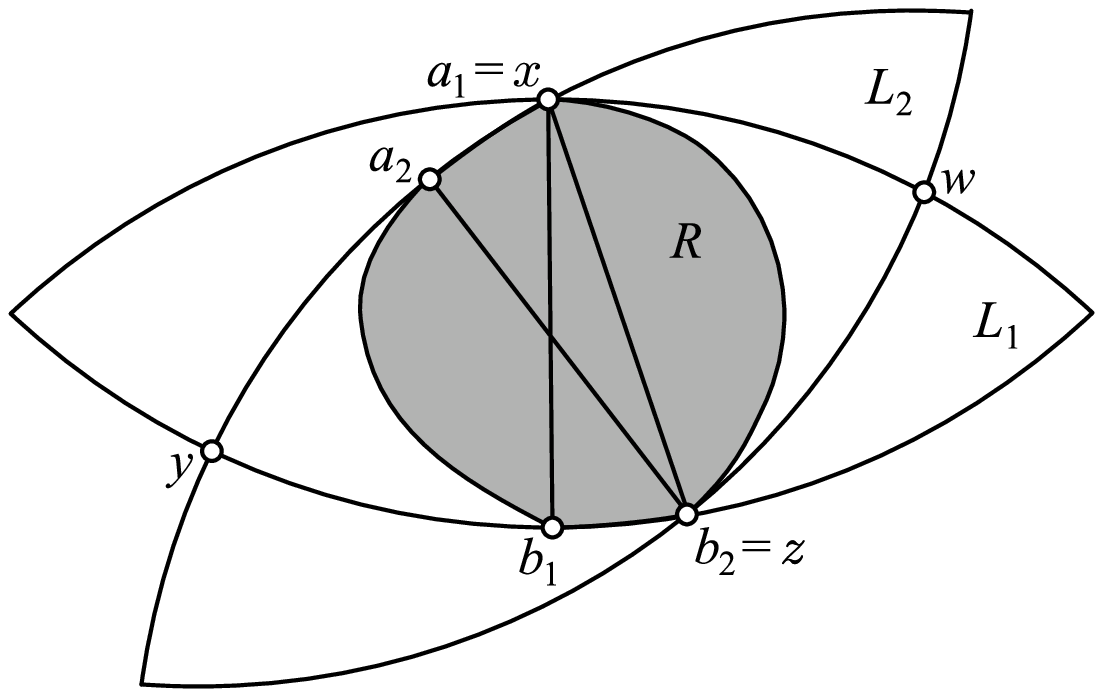} \\ 

\medskip
{FIGURE 1. Illustration to the proof of Theorem 1.}

\end{center}
\medskip






Let $x, y, z, w$ be the vertices of the spherical quadrangle $L_1 \cap L_2$ such that $a_1 \in wx, a_2 \in xy, b_1 \in yz$ and $b_2 \in zw$.
Denote by $H_a$ the hemisphere supporting $R$ at every point of the arc $a_1a_2$ and by $H_b$ the hemisphere supporting $R$ at every point of the arc $b_1b_2$.
Observe that $H_a$ is the left supporting hemisphere of $R$ at $a_1$ and the right supporting hemisphere of $R$ at $a_2$. 
Analogously, $H_b$ is the left supporting hemisphere of $R$ at $b_1$ and the right supporting hemisphere of $R$ at $b_2$. 
From Lemma \ref{leftright} we know that at least one of the hemispheres $M_1$ and $M_1^*$ is a left supporting hemisphere of $R$ at a boundary point of $R$.
Therefore $M_1=H_a$ or $M_1^*=H_b$, and for the same reason we have $M_2=H_a$ or $M_2^*=H_b$.

We cannot have simultaneously $M_1=H_a$ and $M_2=H_a$ because $M_1$ is different from $M_2$. For a similar reason it cannot be simultaneously $M_1^*=H_b$ and $M_2^*=H_b$. Hence either $M_1=H_a$ and $M_2^*=H_b$,  or  $M_1^*=H_b$ and $M_2=H_a$. As a consequence either $a_1 = x$ and $b_2 = z$,  or $a_2 = x$  and $b_1 = z$.

Let for instance $a_1 = x$ and $b_2 = z$ (see Figure 1).  
By the assumptions of our theorem we have $|a_1b_1| = \Delta (R) = |a_2b_2|$. 
Moreover, the triangles $a_1a_2b_2$ and $b_2b_1a_1$ have the common side  $a_1b_2$ and equal angles $\angle a_1a_2b_2 = \frac{\pi}{2} = \angle b_2b_1a_1$ (see Figure 1). 
Hence they are equal.
Thus $|a_1a_2| = |b_1b_2|$.
\end{proof}

Theorem \ref{main} allows us to describe the structure of the boundary of every reduced body on $S^2$.
Namely, we conclude that the boundary consists of in a sense ``opposite" arcs of equal length (some pairs of successive two such arcs may form longer arcs) and from some ``opposite" pieces of spherical ``curves of constant width".
We obtain these ``opposite" pieces of spherical curves of constant width always when ${\rm width}_M (R) = \Delta (R)$ for all $M$ fulfilling $\preceq \!M_1MM_2$, where $M_1$ and $M_2$ are two fixed supporting semicircles of $R$.

\medskip


\begin{pro}\label{a1=a2} 
Let $R \subset S^2$ be a reduced body with $\Delta (R) < \frac {\pi}{2}$.
Assume that ${\rm width}_{M_1} (R) = \Delta (R) = {\rm width}_{M_2} (R)$, where $M_1$ and  $M_2$ are two fixed supporting hemispheres of $R$. 
Denote by $a_i$ the center of the semicircle $M_i/M_i^*$, and by $b_i$ the center of the semicircle $M_i^*/M_i$ for $i=1,2$.
 Assume that $a_1=a_2$  and that $\preceq \!M_1MM_2$ for every $M$ supporting $R$ at this point.
Then the shorter piece of the spherical circle with the center $a_1 = a_2$ and radius $\Delta (R)$ connecting $b_1$ and $b_2$ is in the boundary of $R$.
Moreover, ${\rm width}_M (R) = \Delta (R)$ for all these $M$. 

\end{pro}

\begin{proof}

For shortness, by $a$ denote the point $a_1=a_2$, and by $P$ the piece of the boundary of $R$ from $b_1$ to $b_2$ according to the positive orientation, so this piece which does not contain $a$. 

Let $x\in P$ be an extreme point  of $R$ different from $b_1$ and $b_2$.
By Theorem 4 of \cite{L2} there exists a lune $L$ of thickness $\Delta(R)$ containing $R$ such that $x$ is the center of one of the semicircles bounding $L$.
The lune $L$ has the form $K\cap K^*$, where $K^*$ is the  (unique, by Part I of Theorem 1 of \cite{L2}) hemisphere supporting $R$ at $x$.
Since $x$ lays strictly between $b_1$ and $b_2$ on the boundary of $R$ and since the different points $b_1,x,b_2$ are the centers of the semicircles $M_1^*/M_1, K^*/K, M_2^*/M_2$, respectively, we conclude that the hemisphere $K^*$ satisfies $\prec M_1^*K^*M_2^*$.  
Thus, by Lemma \ref{between} we obtain $\prec M_1KM_2$ and, as a consequence, $a$ is the only point of support of $R$ by $K$.
By Claim 2 of \cite{L2} the center of $K/K^*$ belongs to $R$, which implies that this center is at the point $a$.
Therefore $|ax|=\Delta (R)$.
Since $x$ is an arbitrary extreme point of $P$, we conclude that every extreme point of $P$ is in the distance $\Delta (R)$ from~$a$.

Assume for a while that $P$ contains an arc $x_1x_2$. 
By the preceding paragraph, $x_1$ and $x_2$ are in the distance $\Delta (R)$ from $a$.
Observe that all points laying strictly inside this arc are in the distance less than $\Delta (R)$ from $a$.
Let $H(o)$ be the hemisphere supporting $R$ such that the arc $oa$ intersects the arc $x_1x_2$ in the point $y$ different from $x_1$ and $x_2$.
Since $|ay|< \Delta (R)$, the disk centered at $o$ with radius $\frac{\pi}{2}-\Delta (R)$ does not touch $R$.
Therefore by Theorem 1 of \cite{L2} the lune $H(o)\cap H(o)^*$ is of thickness less than $\Delta (R)$.
It contradicts the fact that $R$ has thickness $\Delta (R)$.
Thus we conclude that $P$ does not contain any arc.

 From the two preceding paragraphs we conclude that all points of $b_1b_2$ are extreme and all of them are in the distance $\Delta (R)$ from $a$. 
Consequently, we obtain the first thesis of our proposition

For the proof of the second part of this proposition take any hemisphere $M$ satisfying $\prec M_1MM_2$. 
Denote by $B$ the disk of radius $\frac{\pi}{2}-\Delta (R)$ concentric with $M$.
By Part I of Theorem 1 from \cite{L2} there is a unique hemisphere $M^*$ and the point of supporting $R$ by $M^*$ is the point where $B$ touches $R$. 
Moreover, $\Delta (M \cap M^*) = \frac{\pi}{2}- \left( \frac{\pi}{2}-\Delta (R)\right) = \Delta (R)$. 
So we get the second part of our proposition. 
\end{proof} 


\begin{thm} \label{segment}
Let $R \subset S^2$ be a reduced body with $\Delta (R) < \frac {\pi}{2}$.
Assume that $M$ is a supporting hemisphere of $R$ such that the intersection of $\bd(M)$ with $\bd (R)$ is a non-degenerated arc  $x_1x_2$. 
Then ${\rm width}_M (R) = \Delta (R)$, and the center of $M/M^*$ belongs to $x_1x_2$.
\end{thm}

\begin{proof}
First, let us show that ${\rm width}_M (R) = \Delta (R)$. 
Assume the opposite, i.e., that ${\rm width}_M (R) > \Delta (R)$.
Then applying the continuity of the width (see Theorem 2 of \cite {L2}), we conclude that there exist supporting hemispheres $M_1$ and $M_2$ of $R$ such that ${\rm width}_{M_1}(R) = \Delta (R) = {\rm width}_{M_2}(R)$ and ${\rm width}_M (R) > \Delta (R)$ for every hemisphere $M$ supporting $R$ strictly between $M_1$ and $M_2$.
Clearly, $M_1$ supports $R$ at $x_1$ and $M_2$ supports $R$ at $x_2$.
Denote by $x_0$ the intersection point of the semicircles $M_1/M_1^*$ and $M_2/M_2^*$.
Since $x_0 \notin x_1x_2$ and $x_1x_2 \subset \bd (R)$, we have $x_0 \notin R$.
But on the other hand, by the result of the paragraph before the last of the proof of Theorem \ref{main} we know that $x_1 = x_0$ or $x_2 = x_0$, which means that $x_0\in R$.
This contradiction implies that ${\rm width}_M (R) = \Delta (R)$.

The second part of our theorem is obvious by $(M/M^*) \cap R = x_1x_2$ and since from Claim 2 in \cite{L2} we know that the center of $M/M^*$ belongs to $R$.
\end{proof}


\begin{thm} 
 If $M$ is an extreme supporting hemisphere of a reduced spherical body $R \subset S^2$, then ${\rm width}_M (R) = \Delta (R)$. 
\end{thm}

\begin{proof}
 Let $M$ be an extreme supporting hemisphere of $R$, say right, at a point $p$.
If the great circle bounding $M$ contains more than one point of $\bd (R)$, we apply Theorem \ref{segment}. 
Otherwise there is a sequence $\left\{p_i\right\}$ of different points of $\bd(R)$ monotonically convergent to $p$ from the  right and such that all the arcs $pp_i$ are not in $\bd(R)$.
The great circle containing $p$ and $p_i$ dissects $R$ into two closed convex subsets.
Denote by $R_i$ this of these subsets which does not contain $p_{i+1}$ and by $M_i$ the hemisphere supporting $R_i$ at  the points of the arc $pp_i$.
Since $R$ is reduced, ${\rm width}_M (R_i) < \Delta (R)$.
The lune $M_i\cap M_i^*$ has thickness ${\rm width}_M (R_i)$ and the limit of  lunes $M_i\cap M_i^*$, as $i$ tends to $\infty$, is the lune $M\cap M^*$. 
Hence the thickness of $M\cap M^*$, which is equal to ${\rm width}_M (R)$, cannot be greater than $\Delta (R)$.
Consequently, it is equal to $\Delta (R)$.
\end{proof}


\begin{pro}\label{hemi} 
Let $p$ be a point of a reduced body $R \subset S^2$ of thickness at most $\frac{\pi}{2}$. 
Then $R \subset H(p)$. 
\end{pro}

\begin{proof} 
If $p$ is an extreme point of $R$, then by the two-dimensional case of Theorem 4 from \cite{L2} there exists a lune $L \supset R$ of thickness $\Delta(R)$ with $p$ as the center of one of the semicircles bounding $L$.
Observe that $L \subset H(p)$.
Hence $R \subset H(p)$, which ends the proof in this case.

If $p$ is a boundary but not an extreme point of $\bd (R)$, then there exist extreme points $e_1, e_2$ of R such that $p \in e_1e_2$.
We already know that $R \subset H(e_i)$ for $i=1,2$.
Thus $R \subset H(e_1) \cap H(e_2)$.
Observe that both the semicircles bounding the lune $H(e_1) \cap H(e_2)$ are contained in $H(p)$. 
Hence this lune is contained in $H(p)$. 
Therefore $R \subset H(p)$, which ends the proof in this case.

If $p$ is not a boundary point of $R$, then there exist boundary points $b_1,b_2$ of $R$ such that $p \in b_1b_2$. 
For analogical reason as in the preceding paragraph, we have $R \subset H(b_1) \cap H(b_2) \subset H(p)$, which ends the proof.
\end{proof}

This Proposition is applied in the proofs of Theorems \ref{constant} and \ref{question}.


\section{Reduced bodies of thickness at least $\frac{\pi}{2}$}

Proposition \ref{smooth} and Theorem \ref{constant} say on the shape of reduced bodies of thickness at least $\frac{\pi}{2}$.
The proof of first of them requires the following lemma.

 
\begin{lem}\label{l_smooth}
 Let $C\subset S^d$ be a spherical convex body with $\Delta (C) > \frac{\pi}{2}$ and let $L \supset C$ be a lune such that $\Delta (L) = \Delta (C)$.
Each of the centers of the $(d-1)$-dimensional hemispheres bounding $L$ belongs to the boundary of $C$ and both they are smooth points of the boundary of $C$.  
\end{lem}

\begin{proof}
We have $L = K \cap M$, where $K$ and $M$ are two hemispheres.
Both of them must support $C$, because in the opposite case we could find a narrower lune than $L$ containing $C$, which, by the assumption $\Delta (L) = \Delta (C)$, is impossible.
For instance, let us show the thesis of our lemma for the $(d-1)$-dimensional hemisphere $M/K$ bounding $L$.
Since $\Delta(C) > \frac{\pi}{2}$, according to Part III of Theorem 1 from \cite{L2} the hemisphere $M$ is a (possibly non-unique) hemisphere fulfilling $M = K^*$. 
Denote by $k$ the center of $K$.
After the just mentioned Part III, the hemisphere $M$ supports $C$ at exactly one point $t$ of the set $\bd(C) \cap B$, where $B$ denotes the largest ball with center $k$ contained in $C$. 
By Corollary 2 of \cite{L2} the point $t$ is the center of $M/K$.
Since $t$ belongs to $\bd(C) \cap B$, it is a smooth point of $\bd(C)$, which ends the proof. 
\end{proof}


\begin{pro}\label{smooth}
 Every reduced spherical body $R \subset S^d$ of thickness over $\frac{\pi}{2}$ is smooth.
\end{pro}

\begin{proof}
In order to see this recall that for every extreme point $e$ of $R$ there exists a lune $L$ 
as in Theorem 4 of \cite{L2} fulfilling $\Delta (L) = \Delta (R)$, so that $e$ is the center of one of the semicircles bounding $L$. 
By Lemma \ref{l_smooth} we obtain that $e$ is a smooth point of the boundary of $R$. 
Since every extreme point of $R$ is a smooth boundary point of $R$, we conclude that $R$ is a smooth spherical convex body.
\end{proof}

The thesis of Proposition \ref{smooth} does not hold true for every thickness at most $\frac{\pi}{2}$.
For instance, on $S^2$ the regular spherical triangle of any thickness at most $\frac{\pi}{2}$ is reduced but not smooth. 


\begin{thm}\label{constant} 
Every reduced spherical convex body $R \subset S^2$ such that $\Delta (R) \geq \frac{\pi}{2}$ is a spherical body of constant width $\Delta (R)$. 
\end{thm}

\begin{proof}
If $\Delta (R) > \frac{\pi}{2}$, we apply Proposition \ref{smooth}.
By Theorem 5 of \cite{L2}, which says that every smooth reduced spherical body is a body of constant width, we obtain the thesis of our theorem. 

Consider the case when $\Delta (R) = \frac{\pi}{2}$.
Assume that there exists a hemisphere $K$ supporting $R$ such that for a fixed $K^*$ the lune $L = K \cap K^*$ has the thickness over $\frac{\pi}{2}$. 
From Corollary 2 in \cite{L2} we obtain that the center of the semicircle $K^*/K$ belongs to $R$.
Denote it by $b$. 

Clearly, there exists an extreme point $e$ of $R$ on the semicircle $K/K^*$.
By Proposition \ref{hemi} the body $R$ is contained in the hemisphere $M$ with the center at $e$.
Thus $|eb|\le \frac{\pi}{2}$. 
But since $\Delta(L) > \frac{\pi}{2}$, the only points of $K/K^*$ lying in the distance at most $\frac{\pi}{2}$ from $b$ are the corners of $L$. 
It means that $e$ is a corner of $L$ and moreover it is the only point of $K/K^*\cap R$.

Observe that since $e$ lies on the boundary of $K$, the lune $K\cap M$ has the thickness $\frac{\pi}{2}$. 
From the definition of ${\rm width}_K (R)$ we obtain that the lune $K\cap K^*$ is of thickness at most $\Delta (K\cap M)=\frac{\pi}{2}$ which contradicts our assumption. 
Thus $R$ is of constant width.
\end{proof}


\label{sec:5}


\section{The case of spherical bodies of constant width}
\label{sec:5}

Recall that every spherical convex body of constant width is reduced. 
Theorem \ref{constant} says that the opposite implication holds true for bodies of thickness at least $\frac{\pi}{2}$ on $S^2$, i.e. every reduced spherical convex body of thickness at least $\frac{\pi}{2}$ 
on $S^2$ is of constant width.
Observe that it does not hold true for reduced spherical convex bodies of thickness below $\frac{\pi}{2}$.
For example every regular spherical triangle of thickness below $\frac{\pi}{2}$ is reduced but is not of constant width. 


\begin{thm}\label{strictly} 
Every spherical strictly convex reduced body of thickness smaller than $\frac{\pi}{2}$ on $S^2$ is of constant width. 
\end{thm}

\begin{proof}
Assume that a spherical strictly convex body $R$ of thickness less than $\frac{\pi}{2}$ is reduced but not of constant width.
It means that there exists a hemisphere $M_0$ for which  ${\rm width}_{M_0} (R) > \Delta (R)$. 
Due to continuity arguments, there exist hemispheres $M_1$ and $M_2$ supporting $R$ such that ${\rm width}_{M_1}(R) = \Delta (R) = {\rm width}_{M_2}(R)$ and ${\rm width}_M (R) > \Delta (R)$ for every hemisphere $M$ supporting $R$ strictly between $M_1$ and $M_2$.
From Theorem \ref{main} we see that the arc with the end-points at the centers of the semicircles $M_1/M_1^*$ and $M_2/M_2^*$ belongs to the boundary of $R$. 
Thus $R$ is not strictly convex. 
\end{proof}

Here is, in a sense, an opposite theorem.


\begin{thm}\label{strictly2} Every spherical convex body of constant width smaller than $\frac{\pi}{2}$ on $S^2$ is strictly convex. 
\end{thm}

\begin{proof}
Assume the opposite, i.e., that there exists a body $W$ of constant width below $\frac{\pi}{2}$ and different extreme points $e_1,e_2$ of $W$ such that the arc $e_1e_2$ belongs to $\bd (W)$.

Denote by $N_1$ the hemisphere supporting $W$ at every point of $e_1e_2$. 
Clearly, its bounding great circle contains $e_1e_2$. 
Since $W$ is of constant width, the lune $L_1 = N_1 \cap N_1^*$ has thickness $\Delta(W)$, and since $W$, as a body of constant width, is a reduced body, by Theorem \ref{segment} the center $c$ of the semicircle $N_1/N_1^*$ belongs to $e_1e_2$. 
Since $c$ must be different from $e_1$ or $e_2$, for instance assume that $c$ is not at $e_1$.

In the preceding and the next paragraph the hemispheres $N_1^*$ and $N_2^*$ supporting $W$ are unique by Part I of Theorem 1 from  \cite{L2}. 
The reason is that $W$ is of constant width below $\frac{\pi}{2}$.

By Theorem 4 of \cite{L2} there exists a hemisphere $N_2$ supporting $W$ at $e_1$ such that $e_1$ is the center of $N_2/N_2^*$.
Assume for a while that $N_1=N_2$. 
From the preceding paragraph we conclude that $N_1^*=N_2^*$ and so $N_1\cap N_1^*= N_2\cap N_2^*$. 
Therefore the center $c$ of $N_1/N_1^*$ coincides with the center $e_1$ of $N_2/N_2^*$. 
It contradicts $c\neq e_1$ and thus we obtain that the hemispheres $N_1$  and $N_2$ must be different.

Denote by $o_i$ the center of $N_i$ for $i=1,2$.
Observe that every hemisphere  supporting $W$ strictly between $N_1$ and $N_2$, supports $W$ only at $e_1$. 
Moreover, the center of every such a hemisphere lays on the arc $o_1o_2$ and vice-versa, every point of this arc is the center of a hemisphere supporting $W$ strictly between $N_1$ and $N_2$.

Let $t,u$ be points in the distance $\Delta (W)$ from $e_1$ such that $t\in e_1o_1$ and $u\in e_1o_2$. 
By Corollary 2  of \cite{L2} we see that $u$ is the center of $N_2^*/N_2$ and it belongs to $W$. 
Observe that $t$ belongs to the disk of radius $\frac{\pi}{2}-\Delta(W)$ with the center at $o_1$. 
By Part I of Theorem 1 of  \cite{L2} this disk touches $W$ only at the center of $N_1^*/N_1$.
Denote this center by $c^*$.
Since the point $c^*$ lays on the arc $co_1$, the point $t$ lays on the arc $e_1o_1$ and the only common point of $co_1$ and $e_1o_1$ is $o_1$, we see that $c^*$ is different  from $t$.
Therefore $t$ does not belong to $W$.

Consider the shorter part $P$ of the spherical circle with the center at $e_1$ and radius $\Delta(W)$ between $t$ and $u$.
Since $W$ is closed, $t \notin W$ and $u \in W$, 
we conclude that there exists a point $x \in P$ different from $t$ and not contained in $W$.
Denote by $o$ the point of the arc $o_1o_2$ such that $x\in oe_1$ and by $N$ the hemisphere centered at $o$.
Since $o$ lays strictly between $o_1$ and $o_2$, the hemisphere $N$ supports $W$ strictly between $N_1$ and $N_2$.
 As a consequence, $N$ supports $W$ only at $e_1$. 
Thus by Claim 2 of \cite{L2} the center of $N/N^*$ is at $e_1$.

 Looking to the proof of Part I of Theorem 1 from \cite{L2}, we observe that the center of our $N^*/N$ lays on the arc connecting the center of $N$ with the center of $N/N^*$ in the distance $\Delta (W)$ from the center of $N/N^*$. 
Therefore $x$ is the center of $N^*/N$ and by Claim 2 of \cite{L2} it belongs to the boundary of $W$. 
But this contradicts the fact that $x$ does not belong to $W$.
 This shows that the assumption from the beginning of our proof must be false.
Hence our theorem is proved. 
\end{proof}

The thesis of this Theorem does not hold true for spherical convex bodies of constant width at least $\frac{\pi}{2}$, as we see from the following Example.

\medskip
{\it Example.} Take a spherical regular triangle $abc$ of sides of length $\kappa < \frac{\pi}{2}$ and prolong them by the same distance $\sigma \leq \frac{\pi}{2} - \kappa$ in both ``directions" up to points $d, e, f, g, h, i$, so that $i, a, b, f$ are on a great circle in this order, that $e, b, c, h$ are on a great circle in this order, and that $g, c, a, d$ are on a great circle in this order. 
Provide three pieces of circles of radius $\kappa + \sigma$: with center $a$ from $f$ to $g$, with center $b$ from $h$ to $i$, with center $c$ from $d$ to $e$. 
Moreover three pieces of circles of radius $\sigma$: with center $a$ from $i$ to $d$, with center $b$ from $e$ to $f$, with center $c$ from $g$ to $h$. 
The convex hull $U$ of these six pieces of circles is a spherical body of constant width $\kappa + 2\sigma$. 
In particular, when $\kappa + \sigma = \frac{\pi}{2}$, three boundary circles of $U$ become arcs; namely $de$, $fg$ and $hi$.

\medskip
In \cite{L2} the question is asked if through every boundary point $p$ of a reduced spherical body $R$ a lune $L \supset R$ of thickness $\Delta(R)$ passes with $p$ as the center of one of the two semicircles bounding $L$. 
The example of the regular spherical triangle of thickness less than $\frac{\pi}{2}$ shows that the answer is negative in general. 
The following proposition shows that for bodies of constant width the answer is positive.






\medskip

\begin{thm} \label{question}
Let $W \subset S^2$ be a spherical body of constant width.
For every boundary point $p$ of $W$ there exists a lune $L \supset W$ fulfilling $\Delta (L) = \Delta(W)$ such that $p$ is the center of one of the semicircles bounding $L$.
\end{thm}

\begin{proof}
If $p$ is an extreme point of $W$, then for this point the thesis is true by Theorem 4 of \cite{L2}.
In particular, the thesis is true for bodies of constant width below $\frac{\pi}{2}$, since then by Theorem \ref{strictly2} all boundary points of $W$ are extreme.

Consider the case when $\Delta(W) = \frac{\pi}{2}$ and $p$ is not an extreme point of $W$.
Clearly, there is a unique hemisphere $G$ supporting $W$ at $p$. 
By Proposition \ref{hemi} we have $W \subset H(p)$. 
Consider the lune $H(p) \cap G$ and note that it contains $W$.
Observe that $p$ is in the distance $\frac{\pi}{2}$ from its 
\break

\begin{center}

\includegraphics[width=3.1in]{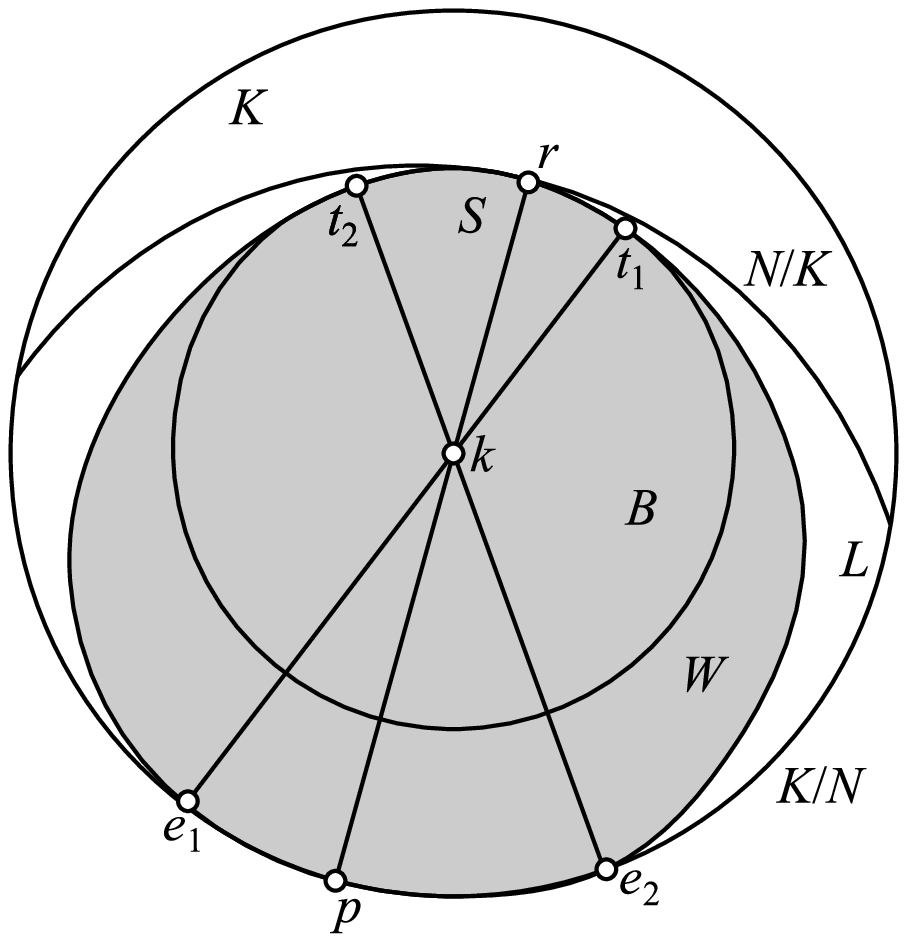} \\ 

{FIGURE 2. Illustration to the proof of Theorem 7.}
\end{center}

\medskip
\noindent
corners, and thus it is the center of a semicircle bounding it. 
Moreover it is in the distance $\frac{\pi}{2}$ from every point of the other semicircle bounding this lune. 
Hence this lune is of thickness $\Delta(W)$.
Thus $H(p) \cap G$ is a lune promised in our theorem.

Finally take into account the case when $\Delta(W) > \frac{\pi}{2}$ and $p$ is not an extreme point of $W$ (the just presented Example shows that it may happen), and let us stay with this assumption up to the end of the proof. 
Clearly, $p$ belongs to a boundary arc $e_1e_2$ of the body $W$, where $e_1$ and $e_2$ are extreme points of $W$ (see Figure 2).

Denote by $K$ the hemisphere supporting $W$ whose bounding great circle contains the arc $e_1e_2$ and by $k$ denote the center of $K$.
From Proposition \ref{smooth} we know that $K$ is the only hemisphere supporting $W$ at each of the points $e_1$ and $e_2$.
By Theorem 4 of \cite{L2} there exist unique hemispheres $K_1^*$ and $K_2^*$ (they play the part of $K^*$ in Theorem 1 of \cite{L2}) supporting $W$ such that the lunes $K \cap K_1^*$ and $K \cap K_2^*$ are of thickness $\Delta (W)$ with $e_1$ as the center of $K/K_1^*$,
and $e_2$ as the center of $K/K_2^*$ ($e_i$ plays the part of $s$ in Part III of Theorem 1 from \cite{L2}). 
Denote by $t_1$ the center of $K_1^*/K$ and by $t_2$ the center of $K_2^*/K$ 
(so $t_i$ plays the part of $t$ in the proof of Part III of Theorem 1 from \cite{L2}).
Clearly the disk $B$ of radius $\rho = \Delta(W) - \frac{\pi}{2}$ and with the center at $k$ (as in Part III of Theorem 1 in \cite{L2}) touches $W$ from inside at the points $t_1$ and $t_2$. 
Moreover, from the proof of Theorem 1 of \cite{L2} and from the earlier established fact that $e_i$ is the center of $K/K_i^*$ and $t_i$ the center of $K_i^*/K$, where $i=1,2$, we obtain that $k$ belongs to both arcs $e_1t_1$ and $e_2t_2$.

Assume for a while that the shorter piece $P$ of $\bd (B)$ between $t_1$ and $t_2$ is not in $\bd (W)$.
Take the corresponding piece of $\bd (W)$ between $t_1$ and $t_2$.  
Clearly, there is an extreme point $x \not \in P$ of this piece.   
The great circle containing $xk$ intersects $P$ at a point $y$, and thus also $e_1e_2$ at a point $q$.
Clearly, $|xq| > |yq| = |yk| + |kq| = \rho + \frac{\pi}{2} = \Delta (W)$. 
By Theorem 4 of \cite{L2} there exist hemispheres $G$ and $H$ such that the lune $J = G \cap H$ of thickness $\Delta (W)$ contains $W$ with $x$ as the center of $H/G$. 
Denote by $z$ the center of $G/H$.
Since $\Delta (J) > \frac{\pi}{2}$, the only farthest point from $x$ on $G$ is just $z$. 
Hence $|xq| \le |xz| = \Delta (J) = \Delta (W)$.
A contradiction with $|xq| > \Delta (W)$ established earlier.
Consequently, $P \subset {\rm bd} (W)$.

Observe that the great circle through $p$ and $k$ intersects $P$ at a point $r$ different from $t_1$ and $t_2$. 
Since $k$ is the center of the hemisphere $K$, we see that $kp$ is orthogonal to the semicircle $K/N$.
Since $k$ is also the center of the disk $B$ and since $r$ is the point at which the great circle $N$ touches $B$, we see that $kr$ is orthogonal to the semicircle $N/K$.
From these two observations we obtain that $pr$ is orthogonal to both the semicircles $K/N$ and $N/K$.
Thus $p$ and $r$ are the centers of the semicircles bounding $L$ (it follows from the obvious fact that the only orthogonal great circle to two different non-opposite meridians is the equator).
Moreover, $L$ is of thickness $|pr| = \Delta (W)$ and hence $L$ is the lune announced in our theorem. 
\end{proof}


\section{Diameter of reduced spherical bodies} 
\label{sec:5}

In the next theorem we prove the conjecture presented in the paragraph before the last of \cite{L3}.
For the proof we need the following two lemmas.
By the way, both may be easily generalized for $S^d$, but for our needs the case when $d=2$ is sufficient.


\begin{lem} \label{max}
Let $L \subset S^2$ be a lune of thickness at most $\frac{\pi}{2}$ whose bounding semicircles are $Q$ and $Q'$.
For every $u, v, z$ in $Q$ such that $v \in uz$ and for every $q \in L$ we have $|qv| \leq \max \{|qu|, |qz| \}$.
\end{lem}

\begin{proof}
If $\Delta(L)= \frac{\pi}{2}$ and $q$ is the center of $Q'$, then the distance between $q$ and any point of $Q$ is the same, and thus the thesis is obvious.
Consider the opposite case, i.e., when $\Delta(L) < \frac{\pi}{2}$, or $\Delta(L)= \frac{\pi}{2}$ but $q$ is not the center of $Q'$.
Clearly, the closest point $p \in Q$ to $q$ is unique.
Observe that for $x \in Q$ the distance $|qx|$ increases as the distance $|px|$ increases. 
This easily implies the thesis of our lemma. 
\end{proof}

For a convex body $C \subset S^2$ denote by ${\rm diam}(C)$ its diameter and by $E(C)$ the set of its extreme points. 


\begin{lem} \label{diamE}
For every spherical convex body $C \subset S^2$ of diameter at most $\frac{\pi}{2}$ we have ${\rm diam} (E(C)) = {\rm diam}(C)$.
\end{lem}

\begin{proof}
Clearly, ${\rm diam} (E(C)) \leq {\rm diam}(C)$.
In order to show the opposite inequality $\diam (C) \leq \diam (E(C))$, thanks to $\diam (C) = \diam (\bd(C))$, it is sufficient to show that $|cd| \leq \diam (E(C))$ for any $c, d \in \bd(C)$. 
If $c, d \in E(C)$, this is trivial.
In the opposite case, at least one of these points does not belong to $E(C)$.
If, say $d \notin E(C)$, then by $d \in \bd(C)$ there are $e,f \in E(C)$ different from $d$ such that $d \in ef$. 
This and $d, e, f \in \bd(C)$ imply that the arc $ef$ is a subset of $\bd(C)$. 

Recall that by Theorem 3 from \cite{L3} we have ${\rm width}_K (C) \leq {\rm diam} (C)$ for every hemisphere $K$ supporting $C$. 
In particular, for the hemisphere supporting $C$ at every point of the arc $ef$. 
Thus by the assumption that ${\rm diam} (C) \leq \frac{\pi}{2}$ we obtain ${\rm width}_K (C) \leq \frac{\pi}{2}$.
Hence we may apply Lemma \ref{max}.
We obtain $|cd| \leq \max \{|ce|, |cf|\}$.

If $c \in E(C)$, from $e, f \in E(C)$ we conclude that $|cd| \leq \diam (E(C))$. 
If $c \notin E(C)$, from $c \in \bd(C)$ we see that there are $g, h \in E(C)$ such that $c \in gh$. 
Similarly as in the consideration of the preceding paragraph we see that $|ec| \leq \max \{|eg|, |eh|\}$ and $|fc| \leq \max \{|fg|, |fh|\}$.
And again by the inequality at the end of the preceding paragraph and by these two just shown inequalities we get $|cd| \leq \max \{|eg|, |eh|, |fg|, |fh|\} \leq \diam (E(C))$, which ends the proof.
\end{proof}

The assumption that ${\rm diam} (C) \leq \frac{\pi}{2}$ is substantial in this lemma, as it follows from the example of any regular triangle of diameter over $\frac{\pi}{2}$. 
The weaker assumption that $\Delta(C) \leq \frac{\pi}{2}$ is not sufficient, which follows from the example of any isosceles triangle $T$ with $\Delta(T) \leq \frac{\pi}{2}$ and the arms longer than $\frac{\pi}{2}$ (so the base shorter than $\frac{\pi}{2}$). 
The diameter of $T$ equals to the distance between the midpoint of the base and the opposite vertex of $T$.
Hence ${\rm diam} (T)$ is over the length of each of the sides.


\begin{thm} \label{ineq} 
For every reduced spherical body $R \subset S^2$ with $\Delta (R) < \frac{\pi}{2}$ we have ${\rm diam}(R) \leq {\rm arc cos} (\cos^2 \Delta (R))$.  
This value is attained if and only if $R$ is the quarter of disk of radius $\Delta(R)$.
If $\Delta (R) \geq \frac{\pi}{2}$, then ${\rm diam}(R) = \Delta (R)$.
\end{thm}

\begin{center}

\includegraphics[width=3.75in]{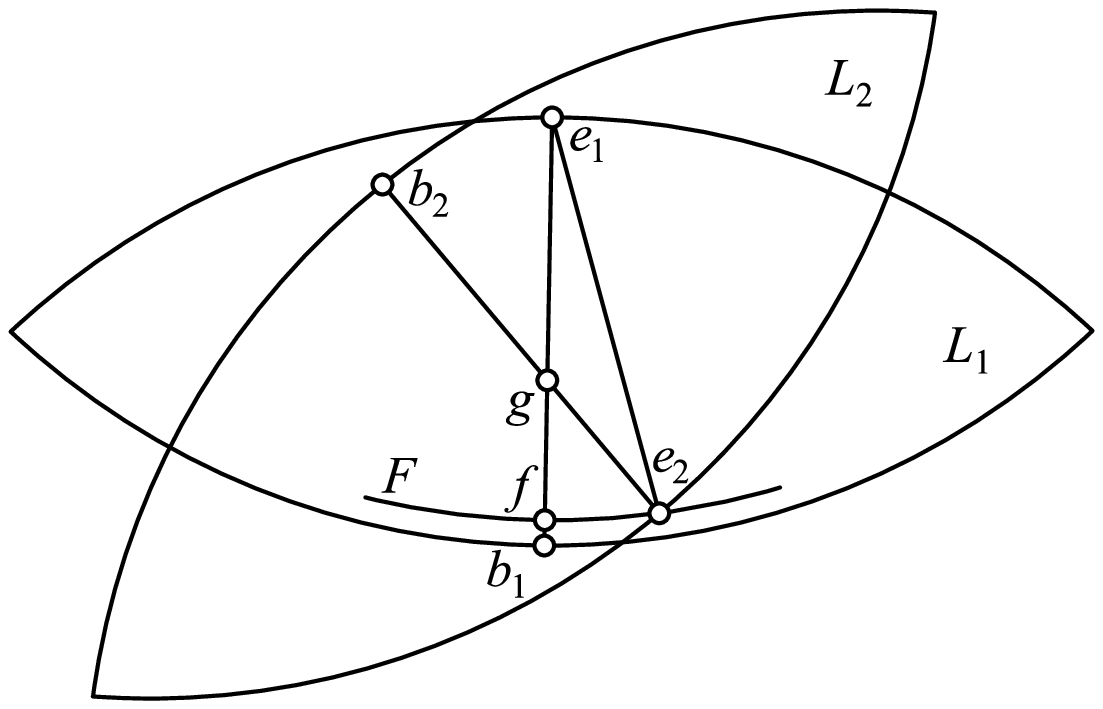} \\ 

\medskip
{FIGURE 3. Illustration to the proof of Theorem 8.}

\end{center}

\begin{proof}
Assume that $\Delta (R) < \frac{\pi}{2}$. 
By Lemma \ref{diamE} it is sufficient to show that the distance between any two points of $E(R)$ is at most ${\rm arc cos} (\cos^2 \Delta (R))$.
Let $e_1, e_2$ be different extreme points of $R$.
Since $R$ is reduced, according to the statement of Theorem 4 from \cite{L2} there exist lunes $L_j \supset R$, where $j=1,2$, of thickness $\Delta(R)$ with $e_j$ as the center of one of two semicircles bounding $L_j$ (see Figure 3). 
Denote by $b_j$, where $j=1,2$, the center of the other semicircle bounding $L_j$.

If the lunes $L_1$ and $L_2$ coincide and $e_1=b_2$, $e_2=b_1$, then 
$|e_1e_2| = \Delta (R) \le \arccos(\cos^2\Delta(R))$.
Otherwise $L_1 \cap L_2$ is a spherical quadrangle with points $e_1,b_2,b_1,e_2$ on the consecutive sides. 
Therefore $e_1b_1$ and $e_2b_2$ intersect at exactly one point. 
Denote it by $g$.
Provide the great circle $F$ orthogonal to $e_1b_1$ which passes through $e_2$.
Since $e_2 \in  L_1$, we see that $F$ intersects $e_1b_1$.
Let $f$ be the intersection point of them. 
From $|e_2b_2| = \Delta (R)$ we see that $|ge_2| \leq \Delta (R)$.
Thus from the right triangle $gfe_2$ we conclude that $|fe_2| \leq \Delta (R)$. 
Moreover, from $|e_1b_1| = \Delta (R)$ and $f \in e_1b_1$ we see that $|fe_1| \leq \Delta (R)$.  
Consequently, from the formula $\cos k = \cos l_1 \cos l_2$ for a right spherical triangle with hypotenuse $k$ and legs $l_1, l_2$ applied to the triangle 
$e_1fe_2$ (again see Figure 3) we obtain $|e_1e_2| \leq {\rm arc cos} (\cos^2 \Delta (R))$. 

Observe that if $\Delta (R)< \frac{\pi}{2}$, then the shown inequality becomes the equality only if $g = f = b_1$ and  $|b_1e_2| = \Delta (R)$. 
In this case, by Proposition \ref{a1=a2}, our body $R$ is a quarter of disk of radius $\Delta(R)$.

 Finally, when we intend to show the last statement of our theorem,
assume that $\Delta (R) \geq \frac{\pi}{2}$.
By Theorem \ref{constant}, the body $R$ is of constant width $\Delta (R)$.

Take a diametral segment $pq$ of $R$. 
Clearly, $p\in \bd (R)$.
Take the lune $L$ from Theorem \ref{question} such that $p$ is the center of a semicircle bounding $L$.
Denote by $s$ the center of the other semicircle $S$ bounding $L$.
By the third part of Lemma 3 of \cite{L2}, we have $|px| < |ps|$ for every $x \not = s$ on $S$.
Hence $|px| \leq |ps|$ for every $x \in S$.
So also $|pz| \leq |ps|$ for every $z \in L$. 
In particular, $|pq|\leq |ps|$.
Since ${\rm diam} (R) = |pq|$ and $\Delta(R) = \Delta(L) = |ps|$, we obtain ${\rm diam} (R) \le \Delta(R)$.

On the other hand, $\Delta (R) \leq {\rm diam} (R)$ by Proposition 1 of \cite{L2}.
Hence the last statement of our theorem is true.\end{proof}

Let us add that an analogical theorem on the diameter of spherical reduced polygons is presented in \cite{L3}.

Now, in this section on the diameter, let us repeat Proposition \ref{hemi} in the following form: {\it for every reduced spherical body with $\Delta (R) \leq \frac{\pi}{2}$ on $S^2$ we have ${\rm diam} (R) \leq \frac{\pi}{2}$}. 
Here is  even a more precise form of it.

\begin{pro} \label{precise}
Let $R \subset S^2$ be a reduced spherical body.
We have $\Delta (R) < {\pi \over 2}$ if and only if ${\rm diam} (R) < {\pi \over 2}$. 
Moreover, $\Delta (R) = {\pi \over 2}$ if and only if ${\rm diam} (R) = {\pi \over 2}$. 
\end{pro}

\begin{proof}
Applying the first derivative test, we see that $f(x) = {\rm arc cos} (\cos^2 x)$ is an increasing function in the interval $\left[0, \frac{\pi}{2}\right]$.
Moreover, for $x = {\pi \over 2}$ this function accepts the value ${\pi \over 2}$.
So in the interval  $[0, {\pi \over 2})$ it accepts only the values below ${\pi \over 2}$.
Thus by Theorem \ref{ineq}, if $\Delta (R) < {\pi \over 2}$, then ${\rm diam} (R) < {\pi \over 2}$.
From the fact that $\Delta (C) \leq {\rm diam} (C)$ for any spherical convex body $C$, which follows from Theorem 3 and Proposition 1 of \cite{L2}, we obtain the opposite implication. 
The second statement is an easy consequence of Proposition \ref{hemi} and the first statement of our proposition. 
\end{proof}

\end{document}